\documentclass[12pt]{amsart}
\usepackage{amsmath,amssymb,amsfonts}
\usepackage{mathtools}
\usepackage[shortlabels]{enumitem}
\setlist{leftmargin=*}

\usepackage[colorlinks=true, linkcolor=blue]{hyperref}

\newtheorem{thm}{Theorem}[section] 

\newtheorem{lem}[thm]{Lemma} 
\newtheorem{prop}[thm]{Proposition}
\newtheorem{claim}[thm]{Claim} 
\newtheorem{fact}[thm]{Fact}
\newtheorem{obs}[thm]{Observation}
\theoremstyle{definition} 

\theoremstyle{remark} 

\newtheorem{rem}[thm]{Remark}

 \numberwithin{equation}{section}

\newcommand{\RR}{\mathbb{R}}
\newcommand{\CC}{\mathbb{C}}
\newcommand{\QQ}{\mathbb{Q}}
\newcommand{\ZZ}{\mathbb{Z}}
\newcommand{\CR}{\mathcal R} 
\newcommand{\NN}{\mathbb N}
\newcommand{\id}{\mathrm{id}}

\def\TT{\mathbf{T}}
\def\Ran{\mathbb{R}_{\mathrm{an}}}

\def\Zt{\widetilde{Z}}
\def\ZF{\Zt_F}
\def\pr{\operatorname{pr}}
\def\St{\operatorname{Stab}}

\begin{document}

\title[A note on o-minimal flows ]{A note on o-minimal flows and   the Ax--Lindemann--Weierstrass theorem for
  abelian varieties over $\mathbb{C}$}

\author{Ya'acov Peterzil}
\address{University of Haifa}
\email{kobi@math.haifa.ac.il}
\author{Sergei Starchenko}
\address{University of Notre Dame}
\email{sstarche@nd.edu}
\thanks{The authors thank  the US-Israel Binational Science Foundation for its
support}
\thanks{The second author was partially supported by  the NSF Research Grant DMS-150067}
\date{November 14, 2016}
\maketitle

\begin{abstract}
In this short note  we present an elementary  proof of Theorem 1.2 from
\cite{flow}, and also  the Ax--Lindemann--Weierstrass theorem for
  abelian and semi-abelian varieties. The proof uses ideas of Pila,
  Ullmo, Yafaev, Zannier (see e.g. \cite{pz}) and is based on
basic properties of sets definable in o-minimal structures. It does
not use the Pila--Wilkie counting theorem.
\end{abstract}

\section{Introduction}

 In their article
\cite{pz}, Pila and Zannier proposed a new method to tackle problems in Arithmetic
geometry, a method which makes use from model theory, and in particular the theory
of o-minimal structures. Their goal was  to produce a new proof for the
Manin-Mumford conjecture and it went roughly as follows:
 Consider the transcendental uniformizing map $\pi:\CC^n\to A$   for an
$n$-dimensional abelian variety $A$. Given an algebraic variety $V\subseteq A$, with
``many'' torsion points, consider its pre-image $\tilde V=\pi^{-1}(V)$. The analytic
periodic set $\tilde V$, when restricted to a fundamental domain $F\subseteq \CC^n$, is
definable in the o-minimal structure $\RR_{an}$. At the heart of the proposed method
was a theorem by Pila and Wilkie, \cite{pw},  used to conclude that $\tilde V$
contains an algebraic variety $X$. In the last step  of the proof one shows that $X$
is contained in a coset of a $\CC$-linear subspace $L$ of $\CC^n$, with $L\subseteq
\tilde V$. Finally, the Zariski closure of $\pi(L)$ is a coset of an abelian
subvariety of $V$, which is the goal of the theorem. Because of various equivalent
formulations this last step of the argument became known as the
``Ax-Lindemann-Weierstrass'' statement for abelian varieties, which we call here
ALW.

 Following the seminal paper of Pila, \cite{Pila1} on the Andre-Oort
Conjecture for $\CC^n$ it became clear that the Pila-Zannier method was very
effective in attacking other problems in arithmetic geometry. Each such  problem was
broken-up into various parts and the ALW was isolated as a separate statement.
Somewhat surprisingly, despite the fact that  ALW does not seem to have a clear
arithmetic content, Pila found an ingenious way to
 apply the Pila-Wilkie theorem again in order to prove it in the setting of the
 Andre-Oort conjecture (this is sometimes called ``the hyperbolic ALW'').
The method of Pila was applied extensively since then to settle several variants of
ALW (\cite{orr}, \cite{pt}, \cite{kuy}).

Our goal in this note is to show that at least in the setting of semi-abelian
varieties, where the uniformizing map is actually a group homomorphism, the use of
Pila-Wilkie is unnecessary and the proof of ALW and most of its variants becomes
quite elementary. We believe that this simpler approach can clarify the picture
substantially and eventually  yield new results as well.

\subsection{Geometric restatements of ALW for semi-abelian varieties}
\label{sec:added-sergei}

The following theorem follows from  a more general  theorem of   Ax (see
\cite[Theorem3]{ax1}) and often is called the full Ax-Lindemann-Weierstrass Theorem.
The original proof of Ax used algebraic differential methods.

\begin{thm}[Full ALW]  Let $G$ be a connected semi-abelian variety
  over $\CC$, $\TT_G$ the Lie algebra of $G$ and $\exp_G\colon \TT_G\to G$ the exponential map.
Let $W\subseteq \TT_G$ be an irreducible
  algebraic variety and $\xi_1,\dotsc, \xi_k\in \CC(W)$.
If    $\xi_1,\dotsc, \xi_k$ are $\QQ$-linearly independent over
$\CC$ then $\exp_G(\xi_1),\dotsc,\exp_G(\xi_k)$ are algebraically
independent over $\CC(W)$.
\end{thm}
If in above theorem we change the conclusion to  ``algebraically
independent over $\CC$'' then we get a weaker statement that often is
called Ax-Lindemann-Weierstrass theorem (ALW theorem for short).

It is not hard to see that both full ALW and ALW theorems can be
interpreted geometrically (see  e.g. \cite{ts} for more details).

\begin{thm}[ALW, Geometric Version]
 Let $G$ be a connected semi-abelian variety
  over $\CC$, $\TT_G$ the Lie algebra of $G$ and $\exp_G\colon \TT_G\to G$ the exponential map.

Let $X\subseteq \TT_G$ be an irreducible
  algebraic variety and $Z\subseteq G$  the Zariski closure of
  $\exp_G(X)$. Then  $Z$ is a translate of an algebraic subgroup of
  $G$.
\end{thm}

We can also restate full ALW.

\begin{thm}[Full ALW, Geometric Version]
Let $G$ be a connected semi-abelian variety
  over $\CC$, $\TT_G$ the Lie algebra of $G$, $\exp_G\colon \TT_G\to
  G$ the exponential map,
and $\pi\colon \TT_G\to \TT_G\times G$ be the map $\pi(z)=(z,\exp_G(z))$.

Let $X\subseteq \TT_G$ be an irreducible
  algebraic variety
and let $Z\subseteq \TT_G\times G$ be the Zariski closure of
  $\pi(X)$. Then  $Z=X\times B$, where B is a translate of an algebraic subgroup of
  $G$.
\end{thm}

\section{Preliminaries}
\label{sec:setting-1}

We work in an o-minimal expansion $\CR$ of the real field $\RR$, and
by definable we always mean  $\CR$-definable (with parameters).
The only property of o-minimal structures that we need is that every
definable discrete subset is finite.

If $V$ is a finite dimensional vector space over $\RR$ and $X$ a
subset of $V$ then, as usual, we say that $X$ is \emph{definable}
if it becomes definable after fixing a basis for $V$ and identifying $V$
with $\RR^n$.  Clearly this notion does not depend on a choice of
basis.

Let  $\pi \colon V\to G$  be a group homomorphisms, where
$V$ is a finite dimensional vector space over $\RR$ and $G$  a
connected commutative algebraic group over $\CC$.  We denote the group
operation of $G$ by $\cdot\,$.

 Let
$\Lambda=\pi^{-1}(e)$.  We say that a subset $F\subseteq V$ is \emph{
a large domain for $\pi$}  if $F$ is a
connected open subset of $V$ with $V=F+\Lambda$. If in addition  the
restriction of $\pi$ to $F$ is definable then we say that $F$ is
\emph{a definable} large domain for $\pi$

\begin{rem}\label{rem:compact} In the above setting if $\pi$ is real analytic and
  $\Lambda$ is a lattice in $V$ then $V/\Lambda$ is compact and there is a
  relatively compact large domain for $\pi$ definable in
  $\Ran$.

\end{rem}

\section{Key observations}
\label{sec:o-minimal-flows}
In this section we fix a
 a finite dimensional $\CC$-vector space $V$, a connected
commutative algebraic group $G$ over $\CC$ and $\pi\colon V\to G$ a complex
analytic group homomorphism.  We assume that $\Lambda=\pi^{-1}(e)$ is
a discrete subgroup of $V$ and  that $\pi$ has a definable
large domain $F$.

Let $X$ be a definable connected real analytic submanifold of $V$ and
let $Z$  be the Zariski
closure of $\pi(X)$ in $G$.

Let $\Zt=\pi^{-1}(Z)$ and  $\ZF=\Zt\cap  F$. The set  $\Zt$ is a
complex analytic $\Lambda$-invariant subset of $V$ and  $\ZF$ is a definable
subset of $F$.

Let
\begin{equation}
  \label{eq:2}
\Sigma_F(X)=\{ v\in V \colon v+X\cap F \neq \emptyset \text{ and }
 v+X\cap F \subseteq \ZF \}.
\end{equation}
Clearly $\Sigma_F(X)$ is a definable subset of $V$.

\medskip
The following is an elementary observation.
\begin{obs}
\label{obs:1}
\begin{enumerate}
\item
If $\lambda\in \Lambda$  and  $\lambda+F\cap X \neq
  \emptyset$  then $-\lambda\in \Sigma_F(X)$.
In particular $X \subseteq F -(\Sigma_F(X)\cap \Lambda)$.
\item
If $v$ is in $\Sigma_F(X)$ then  $v+X\subseteq \Zt$.
\end{enumerate}
\end{obs}

As a consequence we have the following claim.

\begin{claim}\label{claim:key}
$\pi(\Sigma_F(X))\subseteq \St_G(Z)=\{ g\in G \colon g{\cdot}Z=Z\}$.
\end{claim}
\begin{proof}
If $v$ is in $\Sigma_F(X)$ then by Observation \ref{obs:1}(2) we have  $X\subseteq \Zt-v$,
and hence $\pi(X)\subseteq \pi(v)^{-1}{\cdot}\pi(\Zt)=\pi(v)^{-1}{\cdot}Z$.  Since $Z$ is the Zariski
closure of $\pi(X)$  and $\pi(v)^{-1}{\cdot}Z$ is a subvariety of $G$  we have  $Z\subseteq \pi(v)^{-1}{\cdot}Z$, hence
$\pi(v)$ is in the stabilizer of $Z$.
  \end{proof}

\begin{rem}\label{rem:complex}
 Both Observation \ref{obs:1} and Claim \ref{claim:key}
  hold for a complex  irreducible algebraic subvariety $X$  of $V$. It
  can be done either by a direct argument  or
  replacing $X$ with the set $X_\mathrm{reg}$ of smooth points on $X$
  and using the fact that $X_\mathrm{reg}$ is a  connected complex
  submanifold of $V$ that  is dense  in $X$.
\end{rem}

We  deduce a slight generalization of Theorem 1.2 from
\cite{flow}.

\begin{prop} Let $\pi\colon V\to G$ be a complex
analytic group homomorphism from a finite dimensional  $\CC$-vector space
$V$ to  a connected commutative algebraic group $G$  over
$\CC$. Let $\Lambda=\pi^{-1}(e)$.
Assume  $\pi$ has   a large definable domain
$F$.

Let $X\subseteq V$ be a definable connected real analytic
submanifold (or an irreducible complex algebraic subvariety) and $Z\subseteq G$ the Zariski closure of $\pi(X)$ in $G$.

If  $X$ is not covered by finitely many
  $\Lambda$-translate of $F$ then $\St_G(Z)$  is infinite.
\end{prop}
\begin{proof}
  If  $X$ is not covered by finitely many
  $\Lambda$-translate of $F$, then by Observation \ref{obs:1}(1) the set
  $\Sigma_F(X)$ is infinite. Since it is also definable,
  $\pi(\Sigma_F(X))$ must be
  also infinite (otherwise $\Sigma_F(X)$ would be an infinite definable discrete subset
  contradicting o-minimality).
\end{proof}

The following proposition is a key in our proof of ALW.

\begin{prop}\label{prop:key}
 Let $G$ be a connected commutative algebraic group over
  $\CC$, $\TT_G$ the Lie algebra of $G$, and $\exp_G\colon \TT_G\to G$
  the exponential map.  Assume $\exp_G$ has a definable large
  domain $F$.

Let $X\subseteq \TT_G$ be a definable real analytic submanifold (or an
irreducible algebraic subvariety), and $\TT_B
< \TT_G$  the Lie algebra of the stabilizer $B$ of the Zariski
closure of $\exp_G(X)$ in $G$.

Then there is a finite set $S\subset \TT_G$ such that
\[ X\subseteq \TT_B+ S+F.\]
\end{prop}
\begin{proof}
Let $\Lambda=\exp_G^{-1}(e)$. It is a discrete subgroup of $\TT_G$.

Let
 $Z\subseteq G$ be the Zariski closure of
$\exp_G(X)$ and  $B$ be the stabilizer of $Z$ in $G$.

We define $\Sigma_F(X)$ as in \eqref{eq:2}.

Let $B^0$ be the connected component of $B$. It is an algebraic
subgroup of $G$, has a finite index in $B$ and with
$\exp_G(\TT_B)=B^0$, where $\TT_B< \TT_G$ is the Lie algebra of $B$.

 We choose $b_1,\dotsc,b_n\in B$ with
$B=\bigcup_{i=1} b_i{\cdot}B^0$, and also choose $h_1,\dotsc,h_n\in
\TT_G$ with $\exp_G(h_i)=b_i$.  We have
\[ \exp_G\Bigl(\bigcup_{i=1}^n (h_i+\TT_B)\Bigr)= B, \]
hence by Claim \ref{claim:key}
$ \exp_G(\Sigma_F(X))  \subseteq  \exp_G\Bigl(\bigcup_{i=1}^n
  (h_i+\TT_B)\Bigr) $ and
\[ \Sigma_F(X) \subseteq \TT_B + \Bigl(\bigcup_{i=1}^n
  (h_i+\Lambda)\Bigr).
\]
Since  $\Lambda$ is a discrete subgroup of $\TT_G$, the set
$\bigcup_{i=1}^n  (h_i+\Lambda)$ is a discrete subset of $\TT_G$. By
o-minimality, since $\Sigma_F(X)$ is definable we obtain that there is
a finite set $S\subseteq \bigcup_{i=1}^n  (h_i+\Lambda)$ with
$\Sigma_F(X) \subseteq \TT_B + S$.
The proposition now follows from Observation \ref{obs:1}(1).
\end{proof}

\begin{rem}
The above proposition immediately implies ALW Theorem for abelian varieties.
Indeed let $G$ be an abelian variety, $\exp_G\colon \TT_G\to G$ the
exponential map, $X\subseteq \TT_G$ an irreducible algebraic
subvariety, $B<G$ the stabilizer of the Zariski closure of $\exp_G(X)$
and $\TT_B < \TT_G$ the Lie algebra of $B$.

Since $G$ is compact, there is a relatively compact fundamental domain
$F$ for $\exp_G$ definable in the o-minimal structure $\Ran$.

Using Proposition \ref{prop:key}, we have that $X\subseteq \TT_B+S+F$,
for some finite $S\subset \TT_G$.  Since $F$ is relatively compact
we obtain  that $X\subseteq \TT_B +K$ for some compact
$K\subseteq \TT_G$.

Let $L$ be a $\CC$-linear subspace of $\TT_G$ complementary to $\TT_B$.
The projection of $X$ to $L$ along $\TT_B$ is bounded. Since $X$ is an
irreducible variety, it has to be a point. It follows then that
$X\subseteq \TT_B+h$ for some $h\in \TT_G$ and $\exp_G(X) \subseteq
\exp_G(h){\cdot}B$.
\end{rem}

\section{Full ALW for semi-abelian varieties}
\label{sec:full-alw-semi}

In this section we prove a general statement that implies full ALW Theorem
and hence also ALW Theorem for semi-abelian varieties.

\begin{prop}\label{prop:gen}
Let $G$ be a connected semi-abelian variety
  over $\CC$, $\TT_G$ the Lie algebra of $G$, $\exp_G\colon \TT_G\to
  G$ the exponential map, $V$ a vector group over $\CC$ and $\pi\colon
  V\oplus \TT_G \to V\times G$  the map  $\pi=\id_V\times \exp_G$.

Let $Y\subseteq V\oplus \TT_G$ be an irreducible
  algebraic variety
and $Z\subseteq \TT_G\times G$  the Zariski closure of
  $\pi(Y)$. Then  $Z=Z_V\times Z_G$,  where  $Z_V$ is a subvariety of
  $V$ and $Z_G$  a translate of an algebraic subgroup of
  $G$.
\end{prop}
\begin{rem} Since $Z$ is the Zariski closure of $Y$, it is easy to see
  that if $Z=Z_V\times Z_G$ then $Z_V$ must be the Zariski  closure
of $\pr_V(Y)$ and $Z_G$ must be the Zariski closure of
$\exp_G(\pr_{\TT_G}(Y))$,
where $\pr_V$ and $\pr_{\TT_G}$ are the projections from
$V\oplus\TT_G$ to $V$ and $\TT_G$ respectively.
\end{rem}

Before proving the propositioin let's remark how it implies both
versions  of
ALW. To get ALW  we take $V$ to be the trivial vector group $0$.
To get full ALW  we take $V=\TT_G$ and $Y\subseteq \TT_G\oplus \TT_G$
the image of $X$ under the diagonal map, i.e.  $Y=\{ (u,u)\in \TT_G\oplus \TT_G \colon u\in X\}$.

\medskip
We now proceed with the proof of Proposition \ref{prop:gen}.

\begin{proof}
  Let $H=V\times G$. It is a commutative algebraic group with the Lie
algebra $\TT_H=V\oplus \TT_G$ and with the exponential map
$\exp_H=\pi$.  Hence  $Z$ is the Zariski closure of
$\exp_H(Y)$.

We denote the group operation of $H$ by $\cdot$, and view $V$ and $G$
as subgroups of $H$. Very often for subsets $S_1\subseteq V$ and
$S_2\subseteq G$ we write $S_1\times S_2$ instead of $S_1\cdot S_2$ to indicate
that in this case $S_1\cdot S_2$ can be also viewed  as  the Cartesian product of
$S_1$ and $S_2$.

Notice that since $\exp_H$ restricted to $V$ is the identity map we
have  $\exp_H^{-1}(e)=\exp_G^{-1}(e)$.

Let $\St_H(Z)$ be the stabilizer of $Z$ in $H$. It is an algebraic
subgroup of $V\times G$. Since $V$ is a vector group and $G$ is a
semi-abelian variety, $\St_H(Z)$ splits as  $\St_H(Z)=V_0\times B$,
where  $V_0<V$ and $B<G$ are algebraic subgroups.

We first show that $Z\subseteq V\times (p\cdot B)$ for some $p\in G$.

\begin{lem}\label{lem:lemma-1}
We have $ Y-h\subseteq
V+\TT_B \text{ for some } h\in \TT_H, $
where $\TT_B<\TT_G$ is the Lie algebra of $B$.
\end{lem}

\begin{proof}[Proof of Lemma]

Since $G$ is a connected  semi-abelian variety
it admits a short exact sequence
\[ e\to G_0 \to G \to A \to e,\]
where $A$ is an abelian variety and $G_0$ is  an algebraic torus,
i.e. an algebraic group isomorphic to
$(\CC^*,\cdot)^k$.

We do a standard decomposition of $\TT_G$.

  Let $d$ be the dimension of $G$ and $k$  the dimension of
  $G_0$.  Let $\Lambda=\exp_G^{-1}(e)$.  It is a discrete subgroup of $\TT_G$
  whose $\CC$-span is $\TT_G$. Also   $\Lambda$ is a
  free abelian group of rank $2d-k$.

Let $\TT_0<\TT_G$ be the Lie algebra of $G_0$. It is a
$\CC$-linear subspace of $\TT_G$ of dimension $k$.
Let $\Lambda_0=\Lambda\cap \TT_0$. It is easy to see that $\Lambda_0$
 is a pure subgroup of  $\Lambda$ (i.e. for $\lambda\in \Lambda$ and
 $n\in \NN$, $n\lambda\in \Lambda_0$ implies $\lambda\in \Lambda_0$) , hence it  has
  a complementary
subgroup $\Lambda_a$ in $\Lambda$, i.e. a subgroup $\Lambda_a$ of $\Lambda$ with $\Lambda=\Lambda_0\oplus
\Lambda_a$.  Let $L_a< \TT_G$  be the $\RR$-span of $\Lambda_a$.

We have  that $\TT_G=\TT_0\oplus L_a$, and $\Lambda_a$ is a lattice in
$L_a$.

The restriction of $\exp_G$ to $\TT_0$ is a complex Lie group
homomorphism from $\TT_0$ onto $G_0$ whose kernel is $\Lambda_0$.
Choosing an appropriate basis for $\TT_0$ and after identifying  $G_0$
with $(\CC^*,\cdot)^k$, we may assume that $\TT_0=\CC^k$ and the
restriction of $\exp_G$ to $\TT_0$ has form
$(z_1,\dotsc,z_k)\mapsto (e^{2\pi i z_1},\dotsc, e^{2\pi i z_k})$.
In particular $\Lambda_0=\ZZ^k$ and the restriction of $\exp_G$ to
$i\RR^k$ is definable in $\RR_\mathrm{exp}$.

From now on we identify $\TT_0$ with $\CC^k$ and  use decompositions
\[\TT_G=
\CC^k\oplus L_a=\RR^k\oplus i\RR^k\oplus L_a
\text{ and }
\TT_H=V\oplus\RR^k\oplus i\RR^k\oplus L_a.
\]

Since both $L_a/\Lambda_a$ and $\RR^k/\ZZ^k$ are compact we can choose
relatively compact large domains $F_a \subseteq L_a$ and $F_0\subseteq \RR^k$ for
$\exp_G{\restriction} L_a$  and $\exp_G{\restriction} \RR^k$
respectively, definable in $\Ran$

It is easy to see that $F_0+ i\RR^k+F_a$  is a large
domain for $\exp_G$  and $F=V+F_0+ i\RR^k+F_a$ is a large domain for
$\exp_H$, both  definable
in $\RR_\mathrm{an,exp}$.

\medskip

Let $\TT_B<\TT_H$ be the Lie algebra of $B$. Since
$\exp_H^{-1}(e)=\exp_G^{-1}(e)=\Lambda$, we apply  Proposition
\ref{prop:key} to $Y$ and $\exp_H$ and get a finite
$S\subset\TT_H$ with
$Y\subseteq \TT_B+S+F$.
Thus we have
\[
Y\subseteq \TT_B+S+F
= V+
\TT_B+S+F_0+i\RR^k +F_a.
\]

Since the  closures of $F_0$ and $F_a$ are compact, we can find a compact
subset $K\subseteq \TT_H$ with $S+F_0 +F_a \subseteq K$, and hence
\begin{equation}
  \label{eq:2}
 Y\subseteq V+\TT_B+i\RR^k +K.
\end{equation}

Let $M= V+\TT_B+i\RR^k$. It is an $\RR$-linear subspace of $\TT_H$. We
first claim that $Y\subseteq M+h$ for some $h\in \TT_H$. Indeed,
using elementary linear algebra it is sufficient to show
that for any $\RR$-linear map $\xi\colon \TT_H\to \RR$ vanishing on
$M$ the image of
$Y$ under $\xi$ is a point.   Let $\xi\colon \TT_H\to \RR$ be an
$\RR$-linear map vanishing on $M$.  From \eqref{eq:2} we obtain  that
$\xi(Y)$ is bounded.  Therefore, since $Y$ is an irreducible algebraic
variety and the map $\bar \xi\colon \TT_H\to \CC$ given by $\bar
\xi\colon z\mapsto \xi(z)-i\xi(z)$
is a  $\CC$-linear map,  the set $\xi(Y)$ must be a point.
Thus we have $Y\subseteq M+h$ for some $h\in \TT_H$.

We will use the following fact that it is not difficult to prove.

\begin{fact} Let $Y' \subseteq \TT_H$ be an irreducible complex
  analytic subset.  If $W\subseteq\TT_H$ is the $\RR$-span
  of $Y'$ (i.e. the smallest $\RR$-linear subspace containing $Y'$)  then $W$ is a $\CC$-linear subspace of $\TT_H$.

In particular if $Y'\subseteq U$ for some $\RR$-linear subspace $U$ of
$\TT_H$ then $Y' \subseteq iU$.
\end{fact}

Applying  the above fact to $Y'=Y-h$ we obtain
\begin{equation}
  \label{eq:5}
Y-h\subseteq M\cap iM= (V+\TT_B+i\RR^k)\cap (V+\TT_B+\RR^k).
\end{equation}

Thus to finish the proof of Lemma,  it remains  to show that
\begin{equation}
  \label{eq:3}
 (V+\TT_B+i\RR^k)\cap (V+\TT_B+\RR^k)= V+\TT_B.
\end{equation}

 Since $B$ is a
semi-abelian subvariety of $G$, the intersection $B_1=B\cap G_0$ is an
algebraic torus with the Lie algebra  $\TT_{B_1}=\TT_B\cap \CC^k$.
Since $B_1$ is an algebraic subtorus of $G_0$, $\TT_{B_1}$ has a
  $\CC$-basis in $\Lambda\cap \CC^k=\ZZ^k\subset \RR^k$.

It follows then that
$\TT_{B_1}$ has form $E\oplus iE$ for some $\RR$-linear subspace $E\subseteq
\RR^k$, and
 hence
\[ \TT_B\cap (\RR^k+ i\RR^k)=E\oplus  iE. \]

We are now ready to show \eqref{eq:3}. Let $\alpha\in
(V+\TT_B+i\RR^k)\cap (V+\TT_B+\RR^k)$.  Then
\[ \alpha=v_1+u_1+w_1=v_2+ u_2+iw_2 \]
 for some $v_1, v_2\in V, u_1,u_2\in \TT_B, w_1,w_2\in \RR^k$.
Since
$\TT_H=V\oplus \TT_G$, we get $v_1=v_2$ , and  $(u_1-u_2)=-w_1+iw_2$.

Thus $-w_1+iw_2\in \TT_B\cap (\RR^k+ i\RR^k)=E\oplus  iE$, so
$w_1,w_2\in E$, and hence $w_1,w_2\in \TT_B$. It implies that
$\alpha\in V+\TT_B$, that shows \eqref{eq:3}.
It finishes the proof of Lemma.
\end{proof}

We choose $p\in G$ with $V\cdot \exp_H(h)=V \cdot p$ and obtain
\[ \exp_H(Y)\subseteq V\times  (p\cdot B) \text{ for some }p\in G. \]
hence
\begin{equation}
  \label{eq:6}
Z \subseteq V\times  (p\cdot B).
\end{equation}

Let $Z_V=\{ v\in V \colon v\times  p\in Z\}$.  It is an algebraic subvariety of $V$
and we claim that $Z=Z_V\times (p\cdot B)$.

If $v\in Z_V$ then $v\cdot p\in Z$, and since $B$ lies in the stabilizer of $Z$ we
have $v\times (p\cdot B)\subseteq Z$. Hence $Z_V\times (p\cdot B) \subseteq Z$.

Let $v\in V, g\in G$ with $v\cdot g\in Z$.  Since $B$ lies in the stabilizer of $Z$
we have $v\times (g\cdot B ) \subseteq Z$. By \eqref{eq:6}, $v\times (g\cdot
B)\subseteq V\times (p\cdot B)$, hence  $g\cdot B=p\cdot B$, $v\cdot p \in Z$, $v\in
Z_V$ and $v\cdot g \in Z_v\times (p\cdot B)$.  It shows that $Z\subseteq
Z_V\times(p\cdot B)$.

\end{proof}

\bibliographystyle{acm}

\bibliography{flow-paper}

\end{document}